\documentclass[11pt,twoside]{article}

\usepackage{mathrsfs,amsfonts,amsmath,amssymb}
\usepackage{latexsym}
\usepackage{comment} 
\newtheorem{satz}{Theorem}
\newtheorem{proposition}[satz]{Proposition}
\newtheorem{theorem}[satz]{Theorem}
\newtheorem{lemma}[satz]{Lemma}

\newtheorem{corollary}[satz]{Corollary}
\newtheorem{remark}[satz]{Remark}

\def\Z{\mathbb {Z}}
\def\F{\mathbb {F}}
\def\E{\mathsf{E}}

\def\I{{\cal I}}
\def\a{\alpha}

\def\d{\delta}

\def\({\big (}
\def\){\big )}

\def\G{\Gamma}

\def\le{\leqslant}
\def\ge{\geqslant}
\def\_phi{\varphi}
\def\eps{\varepsilon}

\def\Gr{{\mathbf G}}

\def\la{\lambda}

\def\T{\mathsf{T}}
\def\Q{\mathsf{Q}}

\def\SL{{\rm SL}}

\def\Aff{{\rm Aff}}
\def\Sid{\mathsf{Sid}}
\def\F{\mathbb {F}}
\def\R{{\mathbb R}}

\author{Shkredov I.D.}
\title{On multiplicative Chung--Diaconis--Graham
process
\footnote{This work is supported by the Russian Science Foundation under grant 19--11--00001.}
}
\date{}
\begin{document}
	\maketitle


\begin{center}
	Annotation.
\end{center}

{\it \small
    We study the lazy Markov chain on $\F_p$ defined as $X_{n+1}=X_n$ with probability $1/2$ and $X_{n+1}=f(X_n) \cdot \eps_{n+1}$, where $\eps_n$ are random variables distributed  uniformly on $\{ \gamma^{}, \gamma^{-1}\}$, $\gamma$ is a primitive root and $f(x) = \frac{x}{x-1}$ or $f(x)=\mathrm{ind} (x)$.   
    Then we show that 
    the mixing time of $X_n$ is $\exp(O(\log p / \log \log p))$. 
    Also, we obtain  an application to an additive--combinatorial question 
    concerning a certain  
    Sidon--type family of sets. 
}
\\

\section{Introduction}
\label{sec:introduction}

The Chung--Diaconis--Graham process \cite{CDG} is the random walk on $\F_p$ (or more generally on $\Z/n\Z$ for composite $n$) defined as 
\begin{equation}\label{f:CDG}
    X_{j+1} = a X_j + \eps_{j+1} \,, 
\end{equation}
where $a\in \F_p^*$  is a fixed residue
and the random variables $\eps_j$ are independent and identically distributed (in the original paper \cite{CDG} the variables $\eps_j$ were distributed uniformly on  $\{-1,0,1\}$ and 
$a=2$). This process was studied extensively, see papers 
\cite{CD},  \cite{CDG}---\cite{Hildebrand}
and so on. 
In our article we are interested in the following characteristic of $X_n$, which is called the {\it mixing time}.
The definition is 
$$
t_{mix}(\varepsilon) := \inf \left\{n ~:~
\max_{A \subseteq \F_p} \left| \mathrm{P} (X_n \in A) - \frac{|A|}{p} \right| \leq 
\varepsilon\right\} \,.
$$
Usually one takes a concrete value of  the parameter $\eps$, e.g., $\eps=1/4$ and below we will say about $t_{mix} := t_{mix} (1/4)$.
Simple random walk on $\F_p$ has the mixing time $t_{mix}$ of order $p^2$, see \cite{Peres} and it was shown in \cite{CDG} (also, see recent paper \cite{EV}) 
that the mixing time of process \eqref{f:CDG} is at most $O(\log p \cdot \log \log p)$.  
Hence the Chung--Diaconis--Graham process gives an example of a speedup phenomenon, i.e., a phenomenon of increasing the time of the convergence.  
In \cite{He} it was studied a more general non--linear version of the Chung–Diaconis–Graham process, defined as 
\begin{equation}\label{f:f_Markov}
    X_{j+1} = f(X_j) + \eps_{j+1} \,, 
\end{equation}
where $f$ is a bijection on $\F_p$. 
In particular, it was proved that for rational functions of bounded degree (defined correctly at poles, see \cite{He}) the mixing time is 
\begin{equation}\label{f:t_mix_He}
    t_{mix}(1/4) = O(p^{1+\eps}) \,, \quad \quad \quad  \forall \eps>0\,. 
\end{equation}
Perhaps, the right answer for  process \eqref{f:f_Markov} is $t_{mix} = O(\log p)$ but it was obtained in the only case $f(x) = 1/x$ for $x\neq 0$ and $f(0)=0$, see \cite{HPX}. 
The proof is based on $\SL_2 (\F_p)$--actions methods from paper \cite{BG}. 
In \cite{CD} it was asked whether other explicit examples of Markov chains with low mixing time could be provided.

Our paper is devoted to a multiplicative form of Chung--Diaconis--Graham process.
Multiplicative variants of the process were studied in \cite{Asci},  \cite{Hildebrand_mult}, \cite{Hildebrand_mult_m},  \cite{Kruglov} and in other papers. 
Consider the family of functions 
\begin{equation}\label{def:f_ab}
    f^{\alpha,\beta}_* (x) = \frac{x}{\alpha x + \beta} \,,
\end{equation}
where $\alpha, \beta \neq 0$. 
Most 
of 
our results below do not depend on  $\alpha, \beta$, so we will not write these  parameters in such cases. 
In Theorems \ref{t:main_intr}, \ref{t:main} we need $f^{\alpha,\beta}_* (x)$ be a bijection, 
so we put  $f^{\alpha,\beta}_* (-\beta/\alpha):= 1/\alpha$. 
In turn Theorems \ref{t:main_intr}, \ref{t:main} not depend on a particular choice of $(\alpha,\beta)$ and one can consider $\alpha =1$, $\beta=-1$, say, and write $f_* (x) := f^{1,-1}_* (x)$. 
Let us formulate a particular case of our main result. 

\begin{theorem}
    Let $p$ be a prime number and $\gamma \in \F_p^*$ be a primitive root.
    Also, let $\eps_{j}$ be the random variables distributed  uniformly on $\{ \gamma^{}, \gamma^{-1}\}$. 
    Consider the lazy Markov chain $0\neq  X_0,X_1,\dots, X_n, \dots$ defined by 
\[
    X_{j+1}=\left\{\begin{array}{ll}
f_* \left(X_{j}\right) \cdot \varepsilon_{j+1} & \text { with probability } 1 / 2\,, \\
X_{j} & \text { with probability } 1 / 2 \,.
\end{array}\right.
\]
    Then for any $c>0$ and any $n = c \exp(\log p/ \log \log p)$ one has 
\[
    \| P_n - U\| := \frac{1}{2} \max_{A \subseteq \F^*_p} \left| \mathrm{P} (X_n \in A) - \frac{|A|}{p-1} \right| \le e^{-O(c)} \,.
\]
The same is true for the chain $X_{j+1} = f_* \left(X_{j}\right) \cdot \varepsilon_{j+1}$, where $\eps_j$ denote the random variables distributed  uniformly on $\{ 1, \gamma^{-1}, \gamma\}$.  
\label{t:main_intr}
\end{theorem}
In other words, the mixing time of our Markov chain is $\exp(O(\log p/ \log \log p))$.
By a similar method we obtain the same bound for another chain with $f_* (x) = \mathrm{ind} (x)$ and for the chain  of form \eqref{f:f_Markov} with $f(x)=\exp(x)$, see Theorem \ref{t:ind,exp} and formulae \eqref{f:ind}, \eqref{f:exp} below.
As a byproduct we show that in the case $f(x)=x^2$ and $p \equiv 3 \pmod 4$ the mixing time of \eqref{f:f_Markov} is, actually, $O(p\log p)$, see Remark \ref{r:f(x)=x^2}.

Our 
approach 
is not analytical as in \cite{He} but it uses some methods from Additive Combinatorics and Incidence Geometry.
In particular, we apply  some results on  growth in the affine group $\Aff(\F_p)$.  The core of our article  has much more in common with papers \cite{BG}, \cite{s_Kloosterman} than with \cite{He} but we extensively use the general line of the proof from this 
paper. 
From additive--combinatorial point of view the main innovation is a series of asymptotic formulae for the incidences of points and lines, which were obtained via the action of  $\Aff(\F_p)$, see the beginning of section \ref{sec:proof}. The author hopes that such formulae 
are interesting in its own right. 
It is well--known see, e.g., \cite{BG}, \cite{Brendan_rich}, \cite{collinear}, \cite{RS_SL2}, \cite{s_asymptotic}, \cite{s_Kloosterman}, \cite{s_Sidon}, \cite{SdZ} that Incidence Geometry and the sum--product phenomenon sometimes work better than classical analytical methods and that is why it is possible to break 
the square--root barrier, which corresponds to natural bound \eqref{f:t_mix_He} (for details see Theorem \ref{t:Chung} and the proofs of Theorems \ref{t:main}, \ref{t:ind,exp}).

It turns out that the same method is applicable to a purely additive--combinatorial question on Sidon sets. 
Sidon sets is a classical subject of Combinatorial Number Theory,   see, e.g., survey \cite{Bryant}.  
Recall that a subset $S$ of an abelian  group $\Gr$ with the group operation $*$ is called $g$--Sidon set if for any $z\neq 1$ the equation $z=x*y^{-1}$, where $x,y\in S$ has at most $g$ solutions. If $g=1$, then we 
arrive to 
the classical definition of Sidon sets \cite{Sidon}. Having an arbitrary set $A\subseteq \Gr$, we write $\Sid^*(A)$ for size of the maximal (by cardinality) Sidon subset of the set $A$. 
It is known  \cite{KSS} (also, see \cite{Semchenkov}) that for any subset $A$ of our abelian group $\Gr$ the following 
estimate takes place 
$$
\Sid^*(A) \gg \sqrt{|A|}
$$
and Klurman and Pohoata \cite{Klurman} asked about possibility to improve the last bound, having {\it two} different operations on a {\it ring} $\Gr$. In \cite{s_Sidon} the author obtains

\begin{theorem}
	Let $A\subseteq \F$ be a set, where $\F = \R$ or $\F = \F_p$ (in the prime field  case suppose, in addition, that $|A|<\sqrt{p}$, say).
	Then there are some absolute constants $c>0$, $K\ge 1$ such that
\begin{equation}\label{f:Sidon_intr}
	\max \{ \Sid^{+}_K (A), \Sid_K^{\times}(A) \} \gg |A|^{1/2+c} \,.
\end{equation}
\label{t:Sidon_intr}
\end{theorem}

On upper  bounds for \eqref{f:Sidon_intr}, see \cite{R-N_W} and \cite{s_Sidon}.
Notice that $\Sid_K^{\times}(A) = \Sid_K^{+}(\log (A))$ and 
$\Sid_K^{+}(A) = \Sid_K^{\times}(\exp (A))$  for $A\subseteq \R^+$, say. 
Hence it is possible to rewrite bound \eqref{f:Sidon_intr} in terms of the only operation. 
We now consider a general question, which was mentioned by A. Warren during  CANT--2021 conference \cite{Warren}.

\bigskip 

{\bf Problem.} 
{\it Let $f,g$ be some `nice' (say, convex or concave) functions. Is it true that for any set $A\subset \R^+$, say, one has 
\[
	\max \{ \Sid^{+}_K (A), \Sid_K^{+}(f(A)) \} \,, 
	    \quad \quad 
	\max \{ \Sid^{\times}_K (A), \Sid_K^{\times}(g(A)) \} 
	\gg |A|^{1/2+c} \,?
\]
Here $c>0$, $K\ge 1$ are some absolute constants. What can be said for  $K$ exactly equals one  and for a certain $c>0$? 
}

\bigskip 

In this paper we obtain an affirmative answer in the case of $g(x)=x+1$ and $f(x)=\exp(x)$, where in the case of $\F_p$ the latter function is defined as $\exp(x) := g^x$ and $g$ is a fixed primitive root.

\begin{theorem}
	Let $A\subseteq \F$ be a set, where $\F = \R$ or $\F = \F_p$ (in the prime field  case suppose, in addition, that $|A|<\sqrt{p}$).
	Then there are some absolute constants $c>0$, $K\ge 1$ such that
\begin{equation}\label{f:Sidon_intr_new}
	\max \{ \Sid^{\times}_K (A), \Sid_K^{\times}(A+1) \} \gg |A|^{1/2+c} \,,
\end{equation}
    and 
\begin{equation}\label{f:Sidon_intr_new1.5}
	\max \{ \Sid^{+}_K (A), \Sid_K^{+}(\exp(A)) \} \gg |A|^{1/2+c} \,,
\end{equation}
	On the other hand, for any integer $k\ge 1$ there is $A \subseteq \F$ with 
\begin{equation}\label{f:Sidon_intr_new2} 
	\max \{ \Sid^{\times}_k (A), \Sid^{\times}_k (A+1) \} \ll k^{1/2} |A|^{3/4} \,.
\end{equation}
\label{t:Sidon_intr_new}
\end{theorem}

We thank Jimmy He for very useful discussions and valuable  suggestions. 

\section{Definitions and preliminaries}
\label{sec:preliminaries}

By $\Gr$ we denote an abelian group.
Sometimes we underline the group operation writing $+$ or $\times$ in  the considered quantities (as the energy, the representation function and so on, see below).  
Let $\F$ be the  field $\R$ or $\F=\F_p = \Z/p\Z$ for a prime $p$. Let $\F^* = \F \setminus \{0\}$.

We use the same capital letter to denote  set $A\subseteq \F$ and   its characteristic function $A: \F \to \{0,1 \}$. 
Given two sets $A,B\subset \Gr$, define  
the {\it sumset} 
of $A$ and $B$ as 
$$A+B:=\{a+b ~:~ a\in{A},\,b\in{B}\}\,.$$
In a similar way we define the {\it difference sets} and {\it higher sumsets}, e.g., $2A-A$ is $A+A-A$. 
We write $\dotplus$ for a direct sum, i.e., $|A\dotplus B| = |A| |B|$. 
For an abelian group $\Gr$
the Pl\"unnecke--Ruzsa inequality (see, e.g., \cite{TV}) holds stating
\begin{equation}\label{f:Pl-R} 
|nA-mA| \le \left( \frac{|A+A|}{|A|} \right)^{n+m} \cdot |A| \,,
\end{equation} 
where $n,m$ are any positive integers. 
It follows from a more general inequality contained in  \cite{Petridis} that  for 
arbitrary sets $A,B,C \subseteq \Gr$ one has 
\begin{equation}\label{f:Petridis}
    |B+C+X| \le \frac{|B+X|}{|X|} \cdot  |C+X| \,,
\end{equation}
where $X\subseteq A$ minimize the quantity $|B+X|/|X|$.
We  use representation function notations like  $r_{A+B} (x)$ or $r_{A-B} (x)$ and so on, which counts the number of ways $x \in \Gr$ can be expressed as a sum $a+b$ or  $a-b$ with $a\in A$, $b\in B$, respectively. 
For example, $|A| = r_{A-A} (0)$.

For any two sets $A,B \subseteq \Gr$ the {\it additive energy} of $A$ and $B$ is defined by
$$
\E (A,B) = \E^{+} (A,B) = |\{ (a_1,a_2,b_1,b_2) \in A\times A \times B \times B ~:~ a_1 - b^{}_1 = a_2 - b^{}_2 \}| \,.
$$
If $A=B$, then  we simply write $\E^{} (A)$ for $\E^{} (A,A)$.
More generally, for sets (functions) $A_1,\dots, A_{2k}$ belonging an arbitrary (noncommutative) group $\Gr$ and 
$k\ge 2$ define the energy 
$\T_{k} (A_1,\dots, A_{2k})$ as 
\[
	\T_{k} (A_1,\dots, A_{2k}) 
	=
\]
\begin{equation}\label{def:T_k}
	=
	 |\{ (a_1, \dots, a_{2k}) \in A_1 \times \dots \times A_{2k} ~:~ a_1 a^{-1}_2  \dots a_{k-1} a^{-1}_k = a_{k+1} a^{-1}_{k+2}  \dots a_{2k-1} a^{-1}_{2k}  \}| \,.
\end{equation}
In the abelian case put for $k\ge 2$  
\begin{equation}\label{def:E_k}
\E^{+}_k (A) = \sum_x r^k_{A-A} (x) = \sum_{\a_1, \dots, \a_{k-1}} |A\cap (A+\a_1) \cap \dots  \cap (A+\a_{k-1})|^2 \,.
\end{equation}
Clearly, $|A|^k \le \E^{+}_k (A) \le |A|^{k+1}$.
Also, we write $\hat{\E}^{+}_k (A) = \sum_x r^k_{A+A} (x)$.

By $\mathrm{ord} (x)$ denote the multiplicative order of an element of $x\in \F^*_p$ and let $\mathrm{ind} (x)$ is defined as $x=g^{\mathrm{ind} (x)}$, where $g$ is a fixed  primitive root of $\F_p^*$.
It is convenient  for us to think that the function $\mathrm{ind} (x)$ takes  values from $1$ to $p-1$ and hence $\mathrm{ind} (x)$ is defined on $\F_p^*$. 
In a similar way, we denote by $\exp(x) : \F^*_p \to \F_p^*$ the function $\exp(x) = g^x$, where $x\in \F^*_p$.
Let 
$\Aff (\F)$ 
be 
the group of transformations $x\to ax+b$, where $a\in \F^*$, $b\in \F$.
Sometimes we write $(a,b)\in \Aff (\F)$ for the map $x\to ax+b$.

The signs $\ll$ and $\gg$ are the usual Vinogradov symbols.
When the constants in the signs  depend on a parameter $M$, we write $\ll_M$ and $\gg_M$. 
All logarithms are to base $2$.
If we have a set $A$, then we will write $a \lesssim b$ or $b \gtrsim a$ if $a = O(b \cdot \log^c |A|)$, $c>0$.
Let us denote by $[n]$ the set $\{1,2,\dots, n\}$.

\bigskip 

We now mention several useful results, which we will appeal  in the text.
We 
start with 
a result from \cite{s_Kloosterman}.

\begin{lemma}
	Let $f_1,\dots,f_{2k} : \Gr \to \mathbb{C}$ be any functions.
	Then 
\begin{equation}\label{f:T_2^k}
	\T^{2k}_{k} (f_1,\dots, f_{2k}) \le \prod_{j=1}^{2k} \T_{k} (f_j) \,,
\end{equation}
    and 
    $\| f \| := \T_k (f)^{1/2k} \ge \| f\|_{2k}$, $k\ge 2$ is a norm of a function $f: \Gr \to \mathbb{C}$.   
\label{l:T_2^k}
\end{lemma}

The next result on collinear quadruples $\Q(A)$  was proved in \cite{collinear}. We rewrite the asymptotic formula for $\Q(A)$ in the following convenient form. 

\begin{lemma}
    Let $A\subseteq \F_p$ be  a set and $f_A (x) = A(x)-|A|/p$. Then 
\[
    \sum_{l \in \Aff(\F_p)} \left| \sum_x f_A(x) f_A(lx) \right|^4 \ll |A|^5 \log |A| \,,
\]
    where the summation over $l$ in the last formula is taken over all affine transformations.
\label{l:collinear}
\end{lemma}

Finally, we need a simplified version of \cite[Theorem 5]{s_Bourgain}.  

\begin{theorem}
    Let $A,B\subseteq \F_p$ be sets, $|AB|\le M|A|$, $k\ge 2$,  and $|B| \gtrsim_k M^{2^{k+1}}$. 
    Then 
\begin{equation}\label{f:upper_T(AB)}
    \T^{+}_{2^k} (A) \lesssim_k M^{2^{k+1}} \left( \frac{|A|^{2^{k+1}}}{p}  +
    |A|^{2^{k+1}-1} \cdot |B|^{-\frac{k-1}{2}}  \right) \,.  
\end{equation}
\label{t:upper_T(AB)}
\end{theorem}

\section{The proof of the main result}
\label{sec:proof}

We start with our counting Proposition \ref{p:counting_collinear}.
Let $\mathcal{P}, \mathcal{L} \subseteq \F_p \times \F_p$ be a set of points and a set of lines, correspondingly.  
The number of 
incidences 
between $\mathcal{P}$  and $\mathcal{L}$ is 
\begin{equation}\label{def:I(P,L)} 
    \mathcal{I} (\mathcal{P},\mathcal{L}) :=|\{(q,l)\in \mathcal{P} \times \mathcal{L}:\,q\in l\}| \,. 
\end{equation}

\begin{proposition}
    Let $A,B \subseteq \F_p$ be sets and $\mathcal{L}$ be a set of affine transformations. 
    Then for any positive integer $k$ one has 
\begin{equation}\label{f:counting_collinear}
    \I (A\times B, \mathcal{L}) - \frac{|A||B||\mathcal{L}|}{p} 
    \ll 
    \sqrt{|A||B| |\mathcal{L}|} \cdot 
    (\T_{2^k} (\mathcal{L}) |A| \log |A|)^{1/2^{k+2}} \,. 
\end{equation}
\label{p:counting_collinear}
\end{proposition}
\begin{proof} 
We have 
$$
    \I (A\times B, \mathcal{L}) = \frac{|A||B||\mathcal{L}|}{p}  +  \sum_{x\in B} \sum_{l\in \mathcal{L}} f_A (l x) =
 \frac{|A||B||\mathcal{L}|}{p} + \sigma \,.   
$$
 To estimate the error term $\sigma$ we use the H\"older inequality several times as in \cite{Brendan_rich}, \cite{RS_SL2} and obtain 
\[
    \sigma^2 \le |B| \sum_{h} r_{\mathcal{L}^{-1}\mathcal{L}} (h) \sum_x f_A (x) f_A (h x) \,, 
\]
and further
\[
    \sigma^{2^{k}} \le |B|^{2^{k-1}}  |A|^{2^{k-1}-1} 
    \sum_{h} r_{(\mathcal{L}^{-1}\mathcal{L})^{2^{k-1}}} (h) \sum_x f_A (x) f_A (h x) \,.
\] 
Finally, applying Lemma \ref{l:collinear} and  the H\"older inequality one more time, we derive
\[
    \sigma^{2^{k+2}} \ll 
    |B|^{2^{k+1}}  |A|^{2^{k+1}-4} \left(\sum_{h} r^{4/3}_{(\mathcal{L}^{-1}\mathcal{L})^{2^{k-1}}} (h) \right)^3 \cdot |A|^5 \log |A| 
    \ll
\]
\[
    \ll 
    |B|^{2^{k+1}}  |A|^{2^{k+1}-4} \T_{2^k} (\mathcal{L}) |\mathcal{L}|^{2^{k+1}} \cdot |A|^5 \log |A| 
\]  
as required. 
$\hfill\Box$
\end{proof}

\bigskip 

The main advantage of bound  \eqref{f:counting_collinear} of Proposition \ref{p:counting_collinear} is that we have an asymptotic  formula for the number of incidences $\I (A\times B, \mathcal{L})$ (and the set $\mathcal{L}$ can be rather small) but not just upper bounds for $\I (\mathcal{P}, \mathcal{L})$ as in \cite{SdZ}. 
An asymptotic formula for the quantity $\I (\mathcal{P}, \mathcal{L})$  was known before in the specific case of large sets (see \cite{Vinh} or estimate \eqref{f:Vinh} below) and in the case  of Cartesian products but with large sets of lines, see \cite{s_asymptotic} and  \cite{SdZ}.

\bigskip

In the next lemma we estimate the energy $\T_k (\mathcal{L})$ for a concrete family of lines 
which will appear in the proofs of the results of our paper. 

\begin{lemma}
    Let $A,B \subseteq \F^*_p$ be sets, 
    and $\mathcal{L} = \{ (a,b) ~:~ a\in A,\, b\in B\} \subseteq \Aff(\F_p)$. 
    Then for any $k\ge 2$ one has 
\begin{equation}\label{f:T_k-E(L)}
    \T_k (\mathcal{L}) \le  |A|^{2k-1} \T^{+}_k (B) \,.
\end{equation}
\label{l:T_k-E(L)}
\end{lemma}
\begin{proof}
    Let us consider the case of even $k$ and for odd $k$ the arguments are similar. 
    One has $\mathcal{L}^{-1}\mathcal{L} = \{ (a/c, (b-d)/c) ~:~ a,c\in A,\, b,d\in B \}$.
    Considering $\T_{2k} (\mathcal{L})$, we arrive to two equations. The first one is 
\begin{equation}\label{f:first_eq}
    \frac{a_1 \dots a_k}{c_1 \dots c_k} 
    =
    \frac{a'_1 \dots a'_k}{c'_1 \dots c'_k} \,.
\end{equation}
If we fix all variables $a_1 \dots a_k, a'_1 \dots a'_k$, $c_1 \dots c_k, c'_1 \dots c'_k \in A$, then the number of the solutions to the second equation is 
$\T^{+}_{2k} (\a_1 B, \dots, \a_{2k} B)$, where $\a_1,\dots, \a_{2k} \in \F_p^*$ are some elements of $A$ depending on the fixed variables. 
The last quantity is at most $\T^{+}_{2k} (B)$ by Lemma \ref{l:T_2^k}. 
Returning to \eqref{f:first_eq}, we obtain 
the required inequality.
$\hfill\Box$
\end{proof}

\bigskip 

Now we can obtain our first driving result.

\begin{theorem}
    Let $A,B, X_1,Y_1,Z_1 \subseteq \F^*_p$ be sets, $A=X Y_1$, $B=X Y_2$,  $|A|=|X||Y_1|/K_*$, $|B|=|X||Y_2|/K_*$, and $|X Z|\le K|X|$, $|ZZ| \le \tilde{K} |Z|$.
    Suppose that $|Z| \ge p^\delta$ for a certain $\delta \gg \log^{-1} \left( \frac{\log p}{\log \tilde{K}} \right)$.
    Then for a certain $k \ll \delta^{-1}$ the following holds 
\begin{equation}\label{f:sol_Sidon}
    |\{ (a,b)\in A \times B ~:~ a := f_* (b) \}| - \frac{K^2 K_*^2 |A||B|}{p} 
    \ll
    K^2 K^2_* \tilde{K} \sqrt{|A||B|} \cdot p^{-\frac{1}{16^{k}}}
    \,.     
\end{equation}
\label{t:sol_Sidon}
\end{theorem}
\begin{proof}
    Let $\sigma$ be the quantity from the left--hand side of \eqref{f:sol_Sidon}. 
    Also, let $Q_1=AZ$, $Q_2 = BZ$. 
    Then $|Q_1|\le |X Z||Y_1| \le K|X| |Y_1|= KK_* |A|$ and, similarly, for $Q_2$.  
    We have 
\[
    |Z|^2 \sigma \le |\{ (q_1,q_2,z_1,z_2)\in Q_1\times Q_2 \times Z^2  ~:~ q_1/z_1 := f_* (q_2/z_2) \}| \,. 
\]
    Using the definition of the function $f_*$, we arrive to the equation 
\begin{equation}\label{f:lines} 
    \frac{q_1}{z_1} = \frac{q_2}{\alpha q_2 + \beta z_2}
    \quad \quad 
    \implies
    \quad \quad
    \frac{z_1}{q_1} - \frac{\beta z_2}{q_2} = \a \,.  
\end{equation} 
    The last equation can be interpreted as points/lines incidences with the set of lines $\mathcal{L}= Z \times Z$, any $l\in \mathcal{L}$ has the form $l:  z_1X-\beta z_2 Y = \a$ and the set of points $\mathcal{P} = Q^{-1}_1 \times Q^{-1}_2$. 
    Applying Proposition \ref{p:counting_collinear}, we obtain  for any  $k$
\[
    \sigma - \frac{|Q_1||Q_2|}{p} \ll 
        |Z|^{-1} \sqrt{|Q_1||Q_2|} \cdot 
    (\T_{2^k} (\mathcal{L}) |Q_1| \log |Q_1|)^{1/2^{k+2}} \,. 
\]
    Using our bounds for sizes of the sets $Q_1,Q_2$, combining with Lemma \ref{l:T_k-E(L)} and Theorem \ref{t:upper_T(AB)}, we get 
\[
    \sigma - \frac{K^2 K^2_* |A||B|}{p} \lesssim K K_* \tilde{K} \sqrt{|A||B|} \cdot \left(K K_* |A| |Z|^{-\frac{k+1}{2}} \right)^{1/2^{k+2}}
\] 
    provided $|Z| \gtrsim_k \tilde{K}^{2^{k+1}}$ and $|Z|^{k+1} \ll p^2$. 
    Taking $|Z|^{k} \sim p$, we satisfy the second condition and obtain
\[
    \sigma - \frac{K^2 K^2_* |A||B|}{p} \ll K^2 K^2_* \tilde{K} \sqrt{|A||B|} \cdot p^{-\frac{1}{16^{k}}} \,.
\] 
    Choosing 
    $k \sim 1/\delta$, we have the condition $|Z|^{k} \sim p$ and the assumption
    $\delta \gg \log^{-1} \left( \frac{\log p}{\log \tilde{K}} \right)$
    implies that the inequality $|Z| \gtrsim_k \tilde{K}^{2^{k+1}}$ takes place.
$\hfill\Box$
\end{proof}

\begin{remark}
    One can increase the generality of  Theorem \ref{t:sol_Sidon} considering different sets $X_1, X_2, Z_1,Z_2$ such that $|X_1 Z_1|\le K_1 |X_1|$, $|X_2 Z_2|\le K_2 |X_2|$ and so on. We leave the proof of this generalization to the interested reader. 
\end{remark}

\begin{corollary}
    Let $g$ be a primitive root and $I,J \subseteq \F^*_p$ be two geometric progressions with the same base $g$ such that 
    \begin{equation}\label{cond:sol_Sidon_cor}
        \exp(C \log p/ \log \log p) \ll |I| = |J| \le p/2 \,,
    \end{equation} 
    where $C>0$ is an absolute constant. 
    Then 
\begin{equation}\label{f:sol_Sidon_cor}
    |\{ (a,b)\in I \times J ~:~ a := f_* (b) \}| \le (1-\kappa) |I|\,,
\end{equation}
    where $\kappa>0$ is an absolute constant. 
\label{c:sol_Sidon}
\end{corollary}
\begin{proof} 
    Let $I = a\cdot \{1,g,\dots, g^n\}$, 
    $J= b\cdot \{1,g,\dots, g^n\}$,  
    where $n=|I|=|J|$. 
    We apply Theorem \ref{t:sol_Sidon} with $A=I$, $B=J$, $Y_1 = \{a\}$, $Y_2 = \{b\}$, $X=\{1,g,\dots, g^n\}$, $K_*= 1$ and $Z = \{1,g,\dots, g^m \}$, where $m=[cn]$, $c=1/4$.
    Then $K\le 1+c$ and $\tilde{K} <2$. 
    By formula \eqref{f:sol_Sidon}, we obtain 
$$
    |\{ (a,b)\in I \times J ~:~ a := f_* (b) \}|
    - \frac{(1+c)^2 |I||J|}{p} 
    \ll
    |I| \cdot p^{-\frac{1}{16^{k}}} \,.
$$
    We have $\frac{(1+c)^2 |I||J|}{p} \le \frac{25}{32} |I|$ because $n\le p/2$.
    Recalling that $k\sim 1/\delta$ and $\delta \gg (\log \log p)^{-1}$, we derive estimate \eqref{f:sol_Sidon_cor} thanks to our condition \eqref{cond:sol_Sidon_cor}. 
    This completes the proof. 
$\hfill\Box$
\end{proof}

\bigskip 

Now we are ready to prove Theorem \ref{t:main_intr} from the introduction, which we formulate in a slightly general form. In our arguments we use some parts of the proof from \cite{He}.  

\begin{theorem}
    Let $p$ be a prime number and $\gamma \in \F_p^*$ be an element of order at least  
    $$
        \exp(\Omega(\log p/\log \log p)) \,.
    $$
    Also, let $\eps_{j}$ be the random variables distributed uniformly on $\{\gamma^{-1}, \gamma \}$. 
    Consider the lazy Markov chain $0\neq X_0,X_1,\dots, X_n, \dots $ defined by 
\[
    X_{j+1}=\left\{\begin{array}{ll}
f_* \left(X_{j}\right) \cdot \varepsilon_{j+1} & \text { with probability } 1 / 2\,, \\
X_{j} & \text { with probability } 1 / 2 \,. 
\end{array}\right.
\]
    Then for an arbitrary  $c>0$ and for any $n = c \exp(\log p/ \log \log p)$ one has 
\[
    \| P_n - U\| := \frac{1}{2} \max_{A \subseteq \F^*_p} \left| \mathrm{P} (X_n \in A) - \frac{|A|}{p-1} \right| \le e^{-O(c)} \,.
\]
    The same is true for the chain $X_{j+1} = f_* \left(X_{j}\right) \cdot \varepsilon_{j+1}$, where $\eps_j$ denote the random variables distributed  uniformly on $\{ 1, \gamma^{-1}, \gamma\}$.  
\label{t:main}
\end{theorem}

Let $P$ be an ergodic Markov chain on a $k$--regular directed 
graph $G=G(V,E)$.
Let $h(G)$ be the Cheeger constant
\begin{equation}\label{def:Cheeger}
    h(G) = \min_{|S| \le |V|/2} \frac{e(S,S^c)}{k|S|} \,,
\end{equation}
where $e(S,S^c)$ is the number of edges between $S$ and the complement of $S$. 
We need a result from \cite{Chung} (a more compact version is \cite[Theorem 4.1]{He}).

\begin{theorem}
    Let $P$ be an ergodic Markov chain on a graph $G=G(V,E)$. 
    Consider the lazy 
    chain $X_0,X_1,\dots, X_n, \dots $
    with transition matrix $(I+P)/2$, and starting from a certain  deterministic $X_0$.
    Then for any $c>0$ and any $n = c h(G)^{-2} \log |V|$ one has 
\[
    \max_{A\subseteq V} \left| \mathrm{P} (X_n \in A) - \frac{|A|}{|V|} \right| \le e^{-O(c)} \,.
\]
\label{t:Chung}
\end{theorem}

In our case $G=G(V,E)$ with $V=\F_p^*$ and $x \to y$ iff $y=f_*(x) \gamma^{\pm 1}$.
Thus our task is to estimate the Cheeger constant of $G$. 
Take any $S$, $|S|\le p/2$ and write $S$ as 
the disjoint union  $S =\bigsqcup_{j\in J} G_j$, where $G_j$
are geometric progressions with step $\gamma^2$. 
Here and below we use the fact that $\F_p^*$ is cyclic, isomorphic to $\Z/(p-1)\Z$ and generated by a fixed primitive root $g$.  
Consider 
$z,z\gamma, z\gamma^2$, where 
$z\in S$ is a right endpoint (if it exists) of some $G_j$. Then $z\gamma^2 \in S^c$ 
and $z,z\gamma^2$ are connected with $f^{-1}_* (z\gamma)$. 
The point $f^{-1}_* (z\gamma)$ belongs either $S$ or $S^c$ but in any case we have an edge between $S$ and $S^c$. 
Let $J=J_0 \bigsqcup J_1$, where for $j\in J_0$ the set $G_j$ has no  the right endpoint and $J_1 = J\setminus J_0$.
Clearly, $|J_0| \le 2|S|/\mathrm{ord}(\gamma)$. 
By the argument above 
\begin{equation}\label{f:est_h1}
    2h(G) \ge \frac{|J_1|}{|S|} \ge \frac{|J|}{|S|} - \frac{2}{\mathrm{ord}(\gamma)} \,.
\end{equation}
We want to obtain another lower bound for $h(G)$, which works better in the case when $J$ is small. 
Put $L=|S|/|J|$ and 
let $\omega \in (0,1)$ be a small parameter, which we will choose later.
One has  $\sum_{j\in J} |G_j| =|S|$ and hence $\sum_{j ~:~ |G_j|\ge \omega L} |G_j| \ge (1-\omega) |S|$. 
Splitting $G_j$ up into intervals of length exactly $L_\omega := \omega L/2$, we see that the rest is at most $(1-2\omega) |S|$.
Hence we have obtained some geometric progressions $G'_i$, $i\in I$, having lengths  $L_\omega$ and step $\gamma^2$ and such that  $\sum_{i\in I} |G'_i| \ge (1-2\omega) |S|$.
Put $S' = \bigsqcup_{i\in I} G'_j$ and let $\Omega = S\setminus S'$, $|\Omega| \le 2\omega |S|$. 
In other words, we have 
$S' = XY$, $|S'| = |X||Y|\ge (1-2\omega)|S|$, where 
$X=[1,\gamma^2, \dots, \gamma^{2(L_\omega-1)}]$ and $Y$ is a certain set of multiplicative shifts.  
Clearly, 
\begin{equation}\label{f:h_2}
    2h(G) \ge \frac{e(S,S^c)}{|S|} \ge 1- \frac{e(S,S)}{|S|} \ge 
    1 - 8 \omega - \frac{e(S',S')}{|S|} \,.
\end{equation}
Put $Z= [1,\gamma^2, \dots, \gamma^{2(L'_\omega-1)}]$, where $L'_\omega = [cL_\omega]$, where $c=1/4$. 
We have $|ZZ|< 2|Z|$. 
Also, by the assumption the element $\gamma$ has order  at least 
$\exp(\Omega(\log p/\log \log p))$.
Using Theorem \ref{t:sol_Sidon} with $K=1+c$, $\tilde{K}=2$,   $k\sim 1/\delta$ and taking $\delta \ge C(\log \log p)^{-1}$ for sufficiently large constant $C>0$, we get 
\[
    \frac{e(S',S')}{|S|} - \frac{25 |S'|}{16 p} \ll p^{-\frac{1}{16^{k}}}
    \le \frac{1}{32} \,.
\]
Recalling that $|S'|\le |S| \le p/2$, we derive 
\[
    \frac{e(S',S')}{|S|} \le \frac{25}{32} + \frac{1}{32} = \frac{13}{16} \,.
\]
Substituting the last formula into \eqref{f:h_2},  taking sufficiently large $p$ and choosing $\omega = 2^{-8}$, say, we have $h(G) \ge 1/32$. 
We need to check the only condition of Theorem \ref{t:sol_Sidon}, namely,  $|Z| \ge p^\delta$.
If not, then 
$$
    |S|/|J| = L \ll L_\omega \ll |Z| < p^\delta \sim \exp (O(\log p /\log \log p)) \,,
$$
and hence $|J| \gg |S| \exp (-O(\log p /\log \log p))$.
But then by \eqref{f:est_h1} and our assumption $\mathrm{ord}(\gamma) = \exp(\Omega(\log p/\log \log p))$, we see that in any case $h(G) \gg \exp (-O(\log p /\log \log p))$. Combining the last bound for the Cheeger constant and Theorem \ref{t:Chung}, we derive $n \le \exp (O(\log p /\log \log p))$. 

The last part of Theorem \ref{t:main} follows by the same method, combining with the arguments from \cite{CD} and \cite[Section 4.3]{He}. 
We need to ensure that the bijection $f_* (f^{-1}_* (\cdot) \gamma) :\F_p^* \to \F_p^*$ has the same form as in \eqref{def:f_ab} (with our usual convention that $f_*(-\beta/\alpha) =1/\alpha$ of course).
It can be check via a direct calculation or thanks to the fact that 
$f_*$ corresponds to the standard action of a  lower--triangular matrix in $\mathrm{GL}_2 (\F_p)$. 
This completes the proof of Theorem \ref{t:main}.  
$\hfill\Box$

\begin{remark}
    Consider lazy Markov chain \eqref{f:f_Markov} with $f(x)=x^2$ and $p\equiv 3 \pmod 4$. 
    Using the same argument as in the proof of Theorem \ref{t:main}, we need to have deal with the equation 
    $y+a = f(x+b) = x^2 +2bx +b^2$, where $a,b$ belong to some arithmetic progression $P$ and $x,y$ are from a disjoint union of $J$  arithmetic progressions, see details in \cite{He} (strictly speaking, now  the stationary distribution is not uniform and, moreover, 
    our graph is not regular which requires to have a modification of definition \eqref{def:Cheeger}).
    Then last equation can be interpreted as points/lines incidences with the set of lines $\mathcal{L}$ of the form  $Y=2bX + (b^2-a)$ and the set of points $\mathcal{P} = (y-x^2,x)$. 
    Using the main result from \cite{Vinh} (also, see \cite{s_asymptotic}), we obtain 
\begin{equation}\label{f:Vinh} 
    \left| \mathcal{I} (\mathcal{P}, \mathcal{L}) - \frac{|\mathcal{P}||\mathcal{L}|}{p} \right| \le \sqrt{|\mathcal{P}||\mathcal{L}|p} \,.
\end{equation}
    By formula \eqref{f:Vinh} and the calculations as above (see details in \cite[Section 4.2]{He}) we have an expander if $|S|/J \sim |P| \gg \sqrt{p}$.
    If the last inequality does not holds, then $J\gg |S|/\sqrt{p}$ and by an analogue of formula \eqref{f:est_h1}, we obtain $h(G) \gg 1/\sqrt{p}$. 
    Hence in view of Theorem \ref{t:Chung}, we see that the mixing time is $O(p\log p)$. 
\label{r:f(x)=x^2} 
\end{remark}

The method of the proof of Theorem \ref{t:main} (and see Remark \ref{r:f(x)=x^2}) allows us to produce easily some lazy Markov chains on $\F^*_p$ with the mixing time $O(p\log p)$, e.g., 
\begin{equation}\label{f:ind}
    X_{j+1}=\left\{\begin{array}{ll}
\mathrm{ind} \left(X_{j}\right) \cdot \varepsilon_{j+1} & \text { with probability } 1 / 2\,, \\
X_{j} & \text { with probability } 1 / 2 
\end{array}\right.
\end{equation}
($X_0 \neq 0$) 
or as in \eqref{f:f_Markov} with $f(x)= \exp(x)$, namely, 
\begin{equation}\label{f:exp}
    X_{j+1}=\left\{\begin{array}{ll}
\exp \left(X_{j}\right) + \varepsilon_{j+1} & \text { with probability } 1 / 2\,, \\
X_{j} & \text { with probability } 1 / 2 \,.
\end{array}\right.
\end{equation}
Indeed, in the first chain we arrive to the equation $ya=\mathrm{ind} (x) + \mathrm{ind} (b)$ and 
in the second one to $y+b=\exp(x) \cdot \exp(a)$.
Both equations correspond to points/lines incidences. 
Let us underline one more time that our functions $\mathrm{ind} (x), \exp(x)$ are defined on $\F_p^*$ but not on $\F_p$. 
In reality, one has much better bound for the mixing time of two Markov chains above.

\begin{theorem}
    Let $p$ be a prime number and $\gamma \in \F^*_p$. Then the mixing time of Markov chain \eqref{f:exp}  is     $\exp(O(\log p/\log \log p))$.
    If, in addition, the order of $\gamma$ is 
    $\exp(\Omega(\log p/\log \log p))$, then 
    the mixing time of Markov chain \eqref{f:ind}  is $\exp(O(\log p/\log \log p))$.
\label{t:ind,exp}
\end{theorem}
\begin{proof} 
    Our arguments
    follow the same scheme as the proofs of Theorem \ref{t:sol_Sidon} and Theorem \ref{t:main}. In both cases we  need to estimate the energy $\T_{2^k}$ of the set of affine transformations $L$ of the form $x\to gx+r$, where coefficients $g\in \G$ and $r\in P$ belongs to a geometric and an arithmetic progression of size $\sqrt{|L|}$, respectively. An application of Lemma \ref{l:T_k-E(L)} is useless because $\T^{+}_{2^k} (P)$ is maximal.
    Nevertheless, we consider the set $L^{-1}L$ and notice that any element of $L^{-1}L$ has the form $x\to g_2/g_1 x + (r_2-r_1)/g_1$, where $g_1,g_2\in \G$, $r_1,r_2 \in P$.
    Now in view of the arguments of Lemma \ref{l:T_k-E(L)}  our task is to estimate $|\G|^{2^{k+1}-1} |P|^{2^{k+1}} \T^{+}_{2^k} (Q/\G)$, where $Q=P-P$.
    Write $W=Q/\G$ and notice that $|Q|<2|P|$. 
    Taking $X\subseteq \G^{-1}$ as in inequality \eqref{f:Petridis} and applying this inequality with $A=\G^{-1}$, $B=\G^{-1}$ and $C=Q$, we see that 
    $$
        |W X| = |Q/\G \cdot X| \le 2|Q/\G| = 2|W| \,.
    $$
    Increasing the constant $2$ to $O(1)$ in the formula above, one can easily assume (or see \cite{TV}) that for a certain $Y$ the following holds $|Y|\ge |\G|/2$. 
    Applying Theorem \ref{t:upper_T(AB)} with $A=W$ and $B=Y$, we obtain 
\[
    \T^{+}_{2^k} (W) \lesssim_k 
    \frac{|W|^{2^{k+1}}}{p}  +
    |W|^{2^{k+1}-1} \cdot |\G|^{-\frac{k-1}{2}}  
    \,.
\]
    Here we need to assume that $|\G| \gtrsim_k 1$. 
    Hence arguing as in Lemma \ref{l:T_k-E(L)} and using the trivial bound $|W| \le |L|$, we get 
\[
    \T_{2^{k+1}} (L) \lesssim_k 
    |\G|^{2^{k+1}-1} |P|^{2^{k+1}}
        \left( \frac{|L|^{2^{k+1}}}{p}  +
    |L|^{2^{k+1}-1} \cdot |\G|^{-\frac{k-1}{2}}  \right) \ll
    |L|^{2^{k+2}} \cdot |\G|^{-\frac{k+5}{2}} \,,
\]
    provided $|L| \gtrsim_k 1$ and $|\G|^{k+3} \ll p^2$. 
    After that we apply the same argument as in the proof of Theorem \ref{t:main}. 
$\hfill\Box$
\end{proof}

\section{Combinatorial  applications}
\label{sec:applications}

We now obtain an application of the developed   technique to Sidon sets and we follow the arguments from  \cite{s_Sidon}.
We need  Lemma 3, Lemma 7 and Theorem 4 from this paper.

\begin{lemma}
	Let $A\subseteq \Gr$ be a set. 
	Then for any $k\ge 2$ one has 
	\begin{equation}\label{f:random_Ek}
		\Sid_{3k-3} (A) \gg \left( \frac{|A|^{2k}}{\E_k (A)} \right)^{1/(2k-1)} \,, 
	\quad \quad \mbox{ and } \quad \quad 
		\Sid_{2k-2} (A) \gg \left( \frac{|A|^{2k}}{\hat{\E}_k (A)} \right)^{1/(2k-1)} \,.
	\end{equation}
	\label{l:random_Ek}
\end{lemma}

\begin{lemma}
	Let $A\subseteq \Gr$ be a set, $A=B+C$, and $k\ge 1$ be an integer.  
	Then
	$$
		\Sid^{}_k (A) \le \min \{ |C| \sqrt{k|B|}  + |B|, |B| \sqrt{k|C|}  + |C| \} \,. 
	$$
\label{l:L_in_sumsets}
\end{lemma}

\begin{theorem}
	Let $A\subseteq \Gr$ be a set, $\delta, \eps \in (0,1]$ be parameters, $\eps \le \delta$.\\
	$1)~$ Then there is $k=k(\d, \eps) = \exp(O(\eps^{-1} \log (1/\d)))$ such that either $\E^{}_k (A) \le |A|^{k+\delta}$ or there is $H\subseteq \Gr$, $|H| \gtrsim  |A|^{\delta(1-\eps)}$, $|H+H| \ll |A|^\eps |H|$ 
	and there exists $Z\subseteq \Gr$,  $|Z| |H| \ll |A|^{1+\eps}$ 
	with 
	$$|(H\dotplus Z) \cap A| \gg |A|^{1-\eps} \,.$$ 
	$2)~$ Similarly, either there is a set $A'\subseteq A$, $|A'| \gg |A|^{1-\eps}$ and $P\subseteq \Gr$, $|P| \gtrsim |A|^\d$ such that for all $x\in A'$ one has  $r_{A-P}(x) \gg |P| |A|^{-\eps}$ or $\E_k (A) \le |A|^{k+\d}$ with $k \ll 1/\eps$.
\label{t:Ek} 
\end{theorem}

To have deal with the real setting we need the famous Szemer\'edi--Trotter Theorem \cite{ST}.

\begin{theorem}
    Let $\mathcal{P}$, $\mathcal{L}$ be finite sets of points and lines in $\R^2$. Then 
$$
    \mathcal{I} (\mathcal{P}, \mathcal{L}) \ll  (|\mathcal{P}||\mathcal{L}|)^{2/3} + |\mathcal{P}| + |\mathcal{L}| \,.
$$
\label{t:SzT}
\end{theorem}

    Now we are ready to prove Theorem  \ref{t:Sidon_intr}. 
	Take any $\d<1/2$, e.g., $\d=1/4$ and let $\eps \le \d/4$ be a parameter, which we will choose later. 
	In view of Lemma \ref{l:random_Ek} we see that $\E^{\times}_k (A) \le |A|^{k+\d}$ implies 
\begin{equation}\label{tmp:08.03_2} 
	\Sid^{\times}_{3k-3} (A) \gg |A|^{\frac{1}{2} + \frac{1-2\d}{2(2k-1)}} = |A|^{\frac{1}{2} + \frac{1}{4(2k-1)}}
\end{equation} 
	and we are done. 
	Here $k=k(\eps)$. 
	Otherwise there is $H\subseteq \F$, $|H| \gtrsim  |A|^{\d(1-\eps)} \ge |A|^{\d/2}$, $|HH| \ll |A|^{\eps} |H|$ 
	and there exists $Z\subseteq \F$, $|Z| |H| \ll |A|^{1+\eps}$ with  $|(H\cdot Z) \cap A| \gg |A|^{1-\eps} \,.$  Here the product of $H$ and $Z$ is direct. 
	Put $A_* = (H\cdot Z) \cap A$, $|A_*| \gg  |A|^{1-\eps}$ and we want to estimate $\E^{\times}_{l+1} (A_* +1)$ or $\hat{\E}_{l+1}^{\times} (A_* +1)$ for large $l$. 
	After that having a good upper bound for $\E^{\times}_{l+1} (A_* +1)$ or $\hat{\E}_{l+1}^{\times} (A_* +1)$,  we apply  Lemma \ref{l:random_Ek} again to find large multiplicative Sidon subset of $A_*$.

	First of all, notice that in view of \eqref{f:Pl-R}, one has
	$$
	|H A^{-1}_*| \le |HH^{-1}||Z| \ll |A|^{2\eps} |H||Z| \ll |A|^{1+3\eps} \,.
	$$
	In other words, the set $A^{-1}_*$ almost does not grow after the multiplication with $H$.
	Let $Q = H A^{-1}_*$, $|Q| \ll   |A|^{1+3\eps}$ and also let $M=|A|^{\eps}$. 
	Secondly, fix any $\lambda \neq 0, 1$. 
	The number of the solutions to the equation $a_1 / a_2 = \lambda$, where $a_1,a_2 \in A_*+1$ does not exceed 
	$$
		\sigma_\la := |H|^{-2} |\{ h_1,h_2\in H,\, q_1,q_2\in Q ~:~ (h_1/q_1 + 1) / (h_2/q_2 +1) = \lambda \}| \,.
	$$ 
	The last equation  has form \eqref{f:lines}, namely, 
    $$
        \frac{h_1}{q_1} - \frac{\lambda h_2}{q_2} = \lambda - 1
    $$
    and can be interpreted as a question about the number of incidences between points and lines.
    For each $\lambda \neq 0,1$ the quantity $\sigma_\lambda$ can be estimated as
	\begin{equation}\label{f:sigma_la}
		\sigma_\la \ll |H|^{-2} \cdot |Q| |H|^{2-\kappa} \ll |A|^{1+3\eps} |H|^{-\kappa} 
	\end{equation}
	similarly 
	to 
	the proof of Theorem \ref{t:sol_Sidon} above (in the case $\F=\R$ the same is true thanks to Theorem \ref{t:SzT}).
	Here $\kappa = \kappa(\d)>0$.
	Indeed, by our assumption $|A|<\sqrt{p}$, Theorem \ref{t:upper_T(AB)}, Proposition  \ref{p:counting_collinear} and Lemma \ref{l:T_k-E(L)}, we have 
	\begin{equation}\label{f:sigma_la'}
	    \sigma_\la - \frac{|Q|^2}{p} \lesssim  |Q| |H|^{-1/2} 
	    (|Q| \T^{+}_{2^r} (H))^{1/2^{r+2}}
	    \lesssim  |Q| \sqrt{M} \left( M^3 |A| |H|^{-\frac{r+1}{2}} \right)^{1/2^{r+2}}
	\end{equation} 
	provided $|H| \gtrsim_r M^{2^{r+1}}$ and $|H|^{r+1} \ll p$. Here $r$ is a parameter and we take $r\sim 1/\d$ to satisfy the second condition.
	To have the first condition just take $\eps {2^{r+1}} \ll \d$ (in other words,  $\eps \le \exp(-\Omega(1/\d))$) and we are done because 
	$|H| \gg |A|^{\d/2}$.

	Further 
	using 
	$|H| \gg |A|^{\d/2}$, $|A_*| \gg |A|^{1-\eps}$ and 
	choosing 
	any $\eps \le  \d \kappa/100$, we obtain after some calculations and formula \eqref{f:sigma_la} that $\sigma_\la  \ll  |A_*|^{1-\d \kappa/4}$.
	Hence taking sufficiently large $l \gg (\d \kappa)^{-1}$, we derive 
\[
	\hat{\E}_{l+1}^{\times} (A_*) = 
	\sum_{\lambda} r^{l+1}_{A_* A_*} (\lambda) \ll  |A_*|^{l+1} + (|A_*|^{1-\d \kappa/2})^l |A_*|^2 \ll |A_*|^{l+1} + |A|^{l+2-\d\kappa l/2} \ll |A_*|^{l+1}\,.
\]
	Applying Lemma \ref{l:random_Ek} and choosing  $\eps \ll l^{-1}$,   we see that 
	$$
		\Sid^\times_{2l} (A) \ge \Sid^\times_{2l} (A_*)
		\gg	
		|A_*|^{\frac{l+1}{2l+1}}
		\gg
		|A|^{\frac{(1-\eps)(l+1)}{2l+1}}
		= 
		|A|^{\frac{1}{2} + \frac{1-2\eps(l+1)}{2(2l+1)}} \gg |A|^{\frac{1}{2}+c} \,, 
	$$
	where $c = c(\d) >0$ is an absolute constant. 
	We have obtained bound \eqref{f:Sidon_intr} of Theorem \ref{t:Sidon_intr}.


    As for estimate \eqref{f:Sidon_intr_new1.5}, we use the same argument as above but now our analogue of the quantity $\sigma_\lambda$ is  
    $\exp(q_1) \exp(h_1) - \exp(q_2) \exp(h_2) = \lambda$, where $q_1,q_2 \in Q=A_*+H$, $h_1, h_2\in H$ (we use the notation above). 
    The last equation can be treated as  points/lines incidences with the set of lines 
    $x \exp(h_1) - y \exp(h_2) = \lambda$, $|\mathcal{L}| = |H|^2$ and the correspondent set of points $\mathcal{P}$ of size $|Q|^2$. Then analogues of bounds \eqref{f:sigma_la}, \eqref{f:sigma_la'} take place and we are done. 
\bigskip

It remains to obtain estimate \eqref{f:Sidon_intr_new2} of
the theorem.
For any sets $X_1,X_2,X_3$ consider the set $R[X_1,X_2,X_3]$
$$
    R[X_1,X_2,X_3] = \left\{ \frac{x_1-x_3}{x_2-x_3} ~:~ x_1,x_2,x_3 \in X,\, x_2 \neq x_3 \right\} \,.
$$
If $X_1=X_2=X_3=X$, then we put $R[X_1,X_2,X_3] = R[X]$. 
One can check that $1-R[X_1,X_2,X_3] = R[X_1,X_3,X_2]$. 
For $\F = \R$ or $\F=\F_p$ 
we put $X=P$, $A=R[X]$, 
where $P = \{1,\dots, n \}$, $\bar{P} = \{-n, \dots, n \}$ and let $n<\sqrt{p}$ in the case of $\F_p$. 
Then $A$ is contained  in $\bar{P} / \bar{P} := B \cdot C$  and in view of Lemma \ref{l:L_in_sumsets} any multiplicative $k$--Sidon   subset of $A$ has size at most $O (\sqrt{k} |A|^{3/4})$
because as one can check $|A| \gg |P|^2$. 
Further $1-A=A$ and hence the same argument is applicable for the set $1-A$. It remains to notice that $\Sid^\times (X) = \Sid^\times (-X)$ for any set $X$. 
Finally, let us make a remark that there is 
an alternative (but may be a little bit harder) way to obtain estimate \eqref{f:Sidon_intr_new2}. 
Indeed,  consider $R[\G]$, where $\G\subseteq \F_p^*$,  $|\G|< \sqrt{p}$ is a multiplicative subgroup (we consider the case $\F=\F_p$, say). 
One can notice that $R[\G] = (\G-1)/(\G-1)$ and repeat the argument above.


\bigskip

\noindent{I.D.~Shkredov\\
Steklov Mathematical Institute,\\
ul. Gubkina, 8, Moscow, Russia, 119991}
\\
and
\\
IITP RAS,  \\
Bolshoy Karetny per. 19, Moscow, Russia, 127994\\
and 
\\
MIPT, \\ 
Institutskii per. 9, Dolgoprudnii, Russia, 141701\\
{\tt ilya.shkredov@gmail.com}


\begin{thebibliography}{99}



\bibitem{Asci}
{\sc C. Asci, }
Generating uniform random vectors. J. Theor. Probab. 14, 333--356 (2001). 

\bibitem{BG}
{\sc J. Bourgain, A. Gamburd, } 
{\em Uniform expansion bounds for Cayley graphs of $SL_2(\F_p)$, } 
Ann. of Math., 167(2):625--642, 2008.


\bibitem{CD}
{\sc S. Chatterjee, P. Diaconis, }
{\em Speeding up Markov chains with deterministic jumps, }
Probab. Theory Related Fields, 178(3-4):1193--1214, 2020.
	
	
\bibitem{Chung}
{\sc F.R.K. Chung, }
{\em Laplacians and the Cheeger inequality for directed graphs, } Ann. Comb. {\bf 9}  (2005), no. 1, 1--19.


\bibitem{CDG}
{\sc F.R.K. Chung, P. Diaconis, R. L. Graham, }
{\em Random walks arising in random number generation, } 
Ann. Probab., 15(3):1148--1165, 1987.


\bibitem{EV}
{\sc S. Eberhard, P. P. Varj\'u, } 
{\em Mixing time of the Chung–Diaconis–Graham random process, } Probability Theory and Related Fields 179.1 (2021): 317--344.

\bibitem{He}
{\sc J. He, }
{\em Markov chains on finite fields with deterministic jumps, } arXiv:2010.10668 


\bibitem{HPX}
{\sc J. He, H.T. Pham, M.W. Xu, }
{\em Mixing time of fractional random walk on finite fields, } arXiv:2102.02781 (2021).


\bibitem{Hildebrand}
{\sc M. Hildebrand, } 
{\em A lower bound for the Chung--Diaconis--Graham random process, } 
Proc. Am. Math. Soc. 137(4), 1479--1487 (2009).


\bibitem{Hildebrand_mult}
{\sc M. Hildebrand, } 
{\em Random Processes of the Form $X_{n+1} = a_n X_n + b_n \pmod p$ where $b_n$ takes on a Single Value, } 
In: Aldous D., Pemantle R. (eds) Random Discrete Structures. The IMA Volumes in Mathematics and its Applications, vol 76. Springer, New York, NY.


\bibitem{Hildebrand_mult_m}
{\sc M. Hildebrand, }
{\em Random Processes of the Form $X_{n+1} = a_n X_n + b_n \pmod p$, } The Annals of Probability (1993): 710--720.



\bibitem{Klurman} 
https://pohoatza.wordpress.com/2021/01/23/sidon-sets-and-sum-product-phenomena/


\bibitem{KSS}
{\sc J. Koml\'os, M. Sulyok, E. Szemer\'edi, } 
{\em Linear problems in combinatorial number theory, } 
Acta Mathematica Academiae Scientiarum Hungarica 26.1--2 (1975): 113--121.


\bibitem{Kruglov}
{\sc I. A. Kruglov, }
{\em Random sequences of the form $X_t+1=a_t X_t+b_t \pmod n$ with dependent coefficients $a_t$, $b_t$, }
Diskr. Mat., 17:2 (2005), 49--55; Discrete Math. Appl., 15:2 (2005), 145--151. 


\bibitem{Peres}
{\sc D. A. Levin, Y. Peres, } 
{\em Markov chains and mixing times, }
AMS, Providence, RI, 2017. Second edition, with contributions by Elizabeth L. Wilmer, with a chapter on “Coupling from the past”\, by James G. Propp and David B. Wilson.


\bibitem{Petridis}
{\sc G. Petridis, }
{\em New proofs of Pl\"unnecke-type estimates for product sets in groups, }
Combinatorica, {\bf 32}:6 (2012),  721--733.

\bibitem{Brendan_rich}
{\sc B. Murphy, }
{\em Upper and lower bounds for rich lines in grids, }
arXiv:1709.10438v1 [math.CO] 29 Sep 2017.



\bibitem{collinear}
{\sc B. Murphy, G. Petridis, O. Roche‐Newton, M.  Rudnev, I. D. Shkredov}
{\em New results on sum‐product type growth over fields, } 
Mathematika, 65(3)  (2019), 588--642.


\bibitem{Bryant}
{\sc K. O'Bryant, }
{\em A complete annotated bibliography of work related to Sidon sequences, } 
arXiv preprint math/0407117 (2004).


\bibitem{R-N_W}
{\sc O. Roche--Newton, A. Warren, }
{\em Additive and multiplicative Sidon sets, }
arxiv.org/abs/2103.13066



\bibitem{RS_SL2}
{\sc M. Rudnev, I.D. Shkredov, }
{\em On growth rate in $\SL_2(\F_p)$, the affine group and sum-product type implications, }
arXiv:1812.01671v3 [math.CO] 26 Feb 2019. 


\bibitem{Semchenkov} 
{\sc A. S. Semchankau, }
{\em Maximal Subsets Free of Arithmetic Progressions in Arbitrary Sets, } 
Math. Notes, 102:3 (2017), 396--402.  




\bibitem{s_Bourgain}
{\sc I.D. Shkredov, }
{\em Some remarks on the asymmetric sum--product phenomenon, } MJCNT, (2018), 101--126, dx.doi.org/10.2140/moscow.2018..101



\bibitem{s_asymptotic}
{\sc I.D. Shkredov, }
{\em On asymptotic formulae in some sum--product questions, } Tran. Moscow Math. Soc, {\bf 79}:2 (2018), 271--334; English transl. Trans. Moscow Math. Society 2018, pp.231--281.


\bibitem{s_Kloosterman}
{\sc I.D. Shkredov, }
{\em Modular hyperbolas and bilinear forms of Kloosterman sums, } JNT, 20 (2021) 182--211.



\bibitem{s_Sidon}
{\sc I.D. Shkredov, }
{\em On an application of higher energies to Sidon sets, }
arXiv:2103.14670 (2021).


\bibitem{Sidon}
{\sc S. Sidon, }
{\em Ein Satz \"{u}ber trigonomietrische Polynome und seine Anwendungen in der Theorie der Fourier--Reihen, }
Math. Annalen 106 (1932), 536--539. 


\bibitem{SdZ} {\sc S. Stevens, F. de Zeeuw,}  {\em An improved point-line incidence bound over arbitrary fields,} Bull. LMS  {\bf 49}: 842--858, 2017.


\bibitem{ST} {\sc E. Szemer\'edi, W.T.  Trotter, } {\em  Extremal problems in discrete geometry}, Combinatorica {\bf 3}(3-4):381--392, 1983. 



\bibitem{TV}
{\sc T.~Tao, V.~Vu, }{\em Additive combinatorics,} Cambridge University Press 2006.


\bibitem{Vinh} {\sc L. A. Vinh,}
{\em The Szemer\'edi-Trotter type theorem and the sum-product estimate in finite fields,}  European J. Combin. {\bf 32}(8): 1177--1181, 2011.


\bibitem{Warren}
{\sc A. Warren, }
{\em Additive and multiplicative Sidon sets, }
report at CANT--2021; http://www.theoryofnumbers.com/cant/


\end{thebibliography}
\end{document}